\documentclass[12pt]{article}
\usepackage[a4paper, margin=2.3cm]{geometry}
\usepackage{amsmath,amssymb,amsthm}
\usepackage{color}
\usepackage{parskip}

\newtheorem{theorem}{Theorem}[section]

\newtheorem{corollary}[theorem]{Corollary}
\newtheorem{lemma}[theorem]{Lemma}
\newtheorem*{definition*}{Definition}

\newcommand\be{\begin{eqnarray*}}
\newcommand\ee{\end{eqnarray*}}
\newcommand\beq{\begin{equation}}
\newcommand\eeq{\end{equation}}

\newcommand\ben{\begin{eqnarray}}
\newcommand\een{\end{eqnarray}}

\def\F{\mathbb{F}}

\begin{document}
\title{Sum-product estimates over arbitrary finite fields}
\author{
Doowon Koh\thanks{Department of Mathematics, Chungbuk National University. Email: {\tt koh131@chungbuk.ac.kr}}
\and
Sujin Lee\thanks{Department of Mathematics, Chungbuk National University. Email: {\tt sujin4432@chungbuk.ac.kr}}
\and    Thang Pham\thanks{Department of Mathematics,  UCSD. Email: {\tt v9pham@ucsd.edu}}
  \and
  Chun-Yen Shen \thanks{Department of Mathematics,  National Taiwan University and National Center for Theoretical Sciences. Email: {\tt cyshen@math.ntu.edu.tw}}}

\date{}
\maketitle
\date{}
\maketitle
\begin{abstract}
In this paper we prove some results on sum-product estimates over arbitrary finite fields. More precisely, we show that for sufficiently small sets $A\subset \mathbb{F}_q$ we have
\[|(A-A)^2+(A-A)^2|\gg |A|^{1+\frac{1}{21}}.\]
This can be viewed as the Erd\H{o}s distinct distances problem for Cartesian product sets over arbitrary finite fields. We also prove that 
\[\max\{|A+A|, |A^2+A^2|\}\gg |A|^{1+\frac{1}{42}}, ~|A+A^2|\gg |A|^{1+\frac{1}{84}}.\]
\end{abstract}
\section{Introduction}
The well-known conjecture of Erd\H{o}s-Szemer\'edi \cite{es} on the sum-product problem asserts that given any finite set $A \subset \mathbb Z$, one has
$$\max\{ |A+A|, |A \cdot A|\} \geq C_{\epsilon} |A|^{2-\epsilon}$$
for any $\epsilon >0,$ where the constant $C_{\epsilon}$ only depends on $\epsilon$ and the sum and product sets are defined as 
$$A+A = \{a+b\colon a, b \in A\},$$
$$A\cdot A=\{ab\colon a,b \in A\}.$$
In other words, it implies that there is no set $A \subset \mathbb Z$ which is both highly additively structured and multiplicatively structured at the same time.  In order to support their conjecture, they proved that
there is a universal constant $c >0$ so that one has
$$\max\{ |A+A|, |A \cdot A|\} \geq |A|^{1+c}.$$

The constant $c$ has been made explicitly and improved over $35$ years. For instance, Elekes \cite{e1} proved that $c=1/4$, which has been improved to $4/3$ by Solymosi \cite{e2}, to $4/3+5/9813$ by Konyagin and Shkredov \cite{e3},  and to $4/3+1/1509$ by Rudnev, Shkredov, Stevens \cite{e4}. The current best known bound
is $4/3+5/5277$ given by Shakan \cite{sh1}.

In 2004, the finite field analogue of this problem has been first studied by Bourgain, Katz, and Tao \cite{bkt}. They showed that given any set $A \subset \mathbb F_{p}$ with $p$ prime and $p^{\delta} < |A| < p^{1-\delta}$ for some $\delta >0,$ one has
$$\max\{ |A+A|, |A \cdot A|\} \geq C_{\delta}|A|^{1+\epsilon},$$ for some $\epsilon=\epsilon(\delta) > 0$. Actually, this result not only proved a sum-product theorem in the setting of finite fields, but it also has been shown that there are many elegant applications in computer science and related fields. We refer readers to \cite{a3, a1, a2} for more details. 

There are many progresses on making explicitly the exponent $\epsilon$.  The current best bound with $\epsilon =1/5+4/305$ is due to Shakan and Shkredov \cite{sha} by employing a point-line incidence bound and the theory of higher energies. We refer readers to \cite{sha, heg1, heg2, heg3, heg4, pham, shen} and references therein for earlier results. 

In recent years, many variants of sum-product problems have been studied intensively. For example, 
by employing the current breakthrough point-plane incidence bound due to Rudnev \cite{R}, it has been shown in \cite{AMRS} that for any set $A\subset \mathbb{F}_p$, suppose that the size of $A$ is sufficiently small compared with the size of the field, then we have 
\begin{equation}\label{e1x}\max\{|A+A|, |A^2+A^2|\}\gg |A|^{8/7}, ~|(A-A)^2+(A-A)^2|\gg |A|^{9/8}, |A+A^2|\gg |A|^{11/10},\end{equation}
where $A^2:=\{x^2\colon x\in A\}$. These exponents have been improved in recent works. More precisely, Pham, Vinh, and De Zeeuw \cite{pham} showed that  $\max\{|A+A|, |A^2+A^2|\}\gg |A|^{6/5}, |A+A^2|\gg |A|^{6/5}$, and Petridis \cite{P} proved that  $|(A-A)^2+(A-A)^2|\gg |A|^{3/2}$. The higher dimensional version of this result can be found in \cite{pham}. 

We note that the lower bound of $(A-A)^2+(A-A)^2$ is not only interesting by itself in sum-product theory, but it also can be viewed as the finite field version of the celebrated Erd\H{o}s distinct distances problem for Cartesian product sets. We refer readers to \cite{A} for recent progresses on this problem for general sets.

In the setting of arbitrary finite fields $\mathbb{F}_q$ with $q$ is a prime power, the problems will become more technical due to the presence of subfields which eliminate the possibility of sum-product type estimates. It has been proved by Li and Roche-Newton \cite{lili} that for $A\subset \mathbb{F}_q\setminus \{0\}$, if $|A\cap cG|\le |G|^{1/2}$ for any subfield $G$ of $\mathbb{F}_q$ and any element $c\in \mathbb{F}_q^*$, then we have
\[\max\{|A+A|, |A\cdot A|\}\gg |A|^{1+\frac{1}{11}}.\]

The purpose of this paper is  to extend estimates in (\ref{e1x}) to the setting of arbitrary finite fields by employing methods in \cite{bkt, lili}. As mentioned before, the presence of subfields in general fields eliminates the sum-product type estimates. Therefore, it is natural to impose a condition which captures the behavior of how the given set $A$ intersects the subfields. Below are our main theorems.
\bigskip
\begin{theorem}\label{thm1}
Let $A\subset \mathbb{F}_q$. If $|A\cap( aG)|\le |G|^{1/2}$ for any subfield $G$ and $a \in \mathbb F_{q}^*$, then
\[|(A-A)^2+(A-A)^2|\gg |A|^{1+\frac{1}{21}}.\]
\end{theorem}
It is worth noting that one can follow the method in \cite{han} and the sum-product result in \cite{lili} to obtain the exponent $|A|^{1+\frac{1}{26}}$. Therefore, in order to get a better exponent, we need to develop more sophisticated methods to prove our results. In our next theorem, we give a lower bound on $\max\{|A+A|, |A^2+A^2|\}$. 
\bigskip
\begin{theorem}\label{thm2}
If $A\subset \mathbb{F}_q$ and it satisfies that $|(A+A)\cap (aG+b)| \le |G|^{1/2}$ for any subfield $G,$  $a \in \mathbb F_{q}^*,$ and $ b\in \mathbb{F}_q,$ then we have
\[\max\{|A+A|, |A^2+A^2|\}\gg |A|^{1+\frac{1}{42}}.\]
\end{theorem}
\bigskip
An application of the Pl\"unnecke inequality to Theorem 1.2, we have the following corollary.
\bigskip
\begin{corollary}
Let $A\subset \mathbb{F}_q.$  If $|(A+A)\cap (aG+b)| \le |G|^{1/2}$ for any subfield $G$ and $a \in \mathbb F_{q}^*$ and $ b\in \mathbb{F}_q$, then 
\[|A+A^2|\gg |A|^{1+\frac{1}{84}}.\]
\end{corollary}
\bigskip
The rest of the papers are devoted to the proofs of Theorems 1.1 and 1.2.  Throughout this paper, we use the notation $f \gg g$ to mean there is an absolute constant $C$ such that $f \geq Cg$. The constant $C$ may vary from line to line, but is always an absolute constant.
\section{Proof of Theorem \ref{thm1}}
To prove Theorem \ref{thm1}, we make use of the following lemmas.
\begin{lemma}[\cite{vt}]\label{lmt1}
Let $X, B_1, \ldots, B_k$ be subsets of $\mathbb{F}_q$. Then we have
\[|B_1+\cdots +B_k|\le \frac{|X+B_1|\cdots |X+B_k|}{|X|^{k-1}}\]
and
\[|B_1-B_2|\le \frac{|X+B_1||X+B_2|}{|X|}.\]
\end{lemma}
\bigskip
\begin{lemma}[\cite{shen}]\label{lmt2}
Let $X, B_1, \ldots, B_k$ be subsets in $\mathbb{F}_q$. Then, for any $0<\epsilon<1$, there exists a subset $X'\subset X$ such that $|X'|\ge (1-\epsilon)|X|$ and
\[|X'+B_1+\cdots+B_k|\le c\cdot \frac{|X+B_1|\cdots|X+B_k|}{|X|^{k-1}},\]
for some positive constant $c=c(\epsilon).$
\end{lemma}
\bigskip
\begin{lemma}[\cite{lili}]\label{lmt3}
Let $B$ be a subset of $\mathbb{F}_q$ with at least two elements, and define $\mathbb{F}_B$ as the subfield generated by $B$. Then there exists a polynomial $P(x_1, \ldots, x_n)$ in $n$ variables with integer coefficients  such that
\[P(B, \ldots, B)=\mathbb{F}_B.\]
\end{lemma}
\bigskip
\begin{lemma}[\cite{lili}]\label{cover}
Let $X$ and $Y$ be additive sets. Then for any $\epsilon \in (0,1)$ there is some constant $C=C(\epsilon)$ such that at least $(1-\epsilon)|X|$ of the elements of $X$ can be covered by
\[C\cdot \min \left\lbrace \frac{|X+Y|}{|Y|}, \frac{|X-Y|}{|Y|}\right\rbrace.\]
translates of $Y$.
\end{lemma}
\bigskip
We are now ready to prove Theorem 1.1.
\begin{proof}[Proof of Theorem \ref{thm1}]
We first define $\Delta:=|(A-A)^2+(A-A)^2|$.  Without loss of generality, we may assume $1, 0 \in A$ by scaling or translating. We now define the ratio set:
\[R(A, A):=\left\lbrace  \frac{a_1-a_2}{a_3-a_4}\colon a_i \in A, a_3 \ne a_4 \right\rbrace.\]
We now consider the following cases:

{\bf Case 1:} $1+R(A, A)\not\subset R(A, A)$.

In this case, there exist $a_1, a_2, b_1, b_2\in A$ such that
\[r:=1+\frac{a_1-a_2}{b_1- b_2}\not\in R(A, A).\]

One can apply Lemma \ref{cover} four times to obtain a subset $A_1\subset A$ with $|A_1|\gg |A|$ such that $2a_1A_1$ can be covered by at most
\[\frac{|2a_1A_1+A^2-a_1^2|}{|A|}\le \frac{|(A-a_1)^2-A^2-A^2|}{|A|}\le \frac{|(A-A)^2-A^2-A^2|}{|A|}\]
translates of $A^2$,  $2b_1A_1$ can be covered by at most
\[\frac{|2b_1A_1+A_2^2-b_1^2|}{|A_2|}\le \frac{|(A-A)^2-A^2-A^2|}{|A|}\]
translates of $A_2^2$, where $A_2$ is a subset of $A$ with $|A_2|\gg |A|$ and
\[|A_2^2+A^2+A^2+A^2|\ll \frac{|A^2+A^2|^3}{|A|^2},\]
which can be obtained by using Lemma \ref{lmt2}, and for any $x\in \{-b_2, -a_2\}$, the set $-2xA_1$ can be covered by at most
\[\frac{|-2xA_1-A^2+x^2|}{|A|}\le \frac{|(A-x)^2-A^2-A^2|}{|A|}\le \frac{|(A-A)^2-A^2-A^2|}{|A|}\]
translates of $A^2$. Applying Lemma \ref{lmt2} again, we have that there exists a subset $A_3\subset A_1$ such that $|A_3|\gg |A_1|$ and
\begin{equation}\label{eq-o-1-1}
|(b_1-b_2)A_3+(b_1-b_2)A_1+(a_1-a_2)A_1|\ll \frac{|A+A||(b_1-b_2)A_1+(a_1-a_2)A_1|}{|A_1|}.\end{equation}
On the other hand, we also have
\begin{equation}\label{eq-o-2-1}|(b_1-b_2)A_3+(b_1-b_2)A_1+(a_1-a_2)A_1| \geq |A_3+rA_1|, \end{equation}
because $r\not\in R(A, A)$ implies that the equation
\[a_1-a_2=r(b_1+rb_2)\]
has no non-trivial solutions.  This gives us
\begin{equation}\label{eq-o-3-1}|A_3+rA_1|=|A_3||A_1|\gg |A|^2. \end{equation}
We now estimate $(b_1-b_2)A_1+(a_1-a_2)A_1$ as follows.

First we note that
\[|(b_1-b_2)A_1+(a_1-a_2)A_1|=|2b_1A_1-2b_2A_1+2a_1A_1-2a_2A_1|.\]

Since $2b_1A_1$ can be covered by at most $|(A-A)^2+A^2-A^2|/|A|$ copies of $A_2^2$, $-2b_2A_1$ can be covered by at most $|(A-A)^2+A^2-A^2|/|A|$ copies of $A^2$, $2a_1A_1$ can be covered by at most $|(A-A)^2+A^2-A^2|/|A|$ copies of $A^2$, and $-2a_2A_1$ can be covered by at most $|(A-A)^2+A^2-A^2|/|A|$ copies of $A^2$, we have
\begin{align}\label{eq-o-4-1}
|(b_1-b_2)A_1+(a_1-a_2)A_1|&\ll \frac{|(A-A)^2-A^2-A^2|^4}{|A|^4}\cdot |A_2^2+A^2+A^2+A^2|\nonumber\\
&\le  \frac{|A^2+A^2|^3}{|A|^6} |(A-A)^2-A^2-A^2|^4.
\end{align}

Lemma \ref{lmt2} tells us that there exists a set $X\subset A^2$ such that $|X|\gg |A|$ and 
\[|X+A^2+A^2|\ll \frac{|A^2+A^2|^2}{|A|}.\]

So, applying Lemma \ref{lmt1}, we have
\begin{align*}
|(A-A)^2-(A^2+A^2)|&\ll \frac{|(A-A)^2+X||X+A^2+A^2|}{|X|} \ll \frac{\Delta^3}{|A|^2}.
\end{align*}

Putting (\ref{eq-o-1-1})-(\ref{eq-o-4-1}) together, and using the fact that $|A^2+A^2|\le \Delta$ and $|A+A| \leq \frac{|A-A|^2}{|A|}$, we have
\[\Delta\gg |A|^{1+\frac{1}{17}}.\]
{\bf Case 2:} $A\cdot R(A, A)\not\subset R(A, A)$.
As above, there are elements $a_1, a_2, b, b_1, b_2\in A$ such that
\[r:=b\cdot \frac{a_1-a_2}{b_1-b_2}\not\in R(A, A).\]
Note that $b\ne 0$ and $a_1\ne a_2$ since $0\in R(A, A)$.  Thus $r^{-1}$ exists.

Let $A_1$ be the set as in Case $1$. Lemma \ref{cover} implies that there exists a set $A_2\subset A_1$ such that $|A_2|\gg |A_1|$ and $-2bA_2$ can be covered by at most
\[\frac{|-2bA_2+A^2+b^2|}{|A|}\le \frac{|(A-b)^2+A^2-A^2|}{|A|}\]
translates of $A^2$.

Using the same argument as above, we have
\begin{align*}
|A|^2\ll |A_2+rA_2|&=|r^{-1}A_2+A_2|\ll \frac{|b^{-1}A_2+A_2||(a_1-a_2)A_2+(b_1-b_2)A_2|}{|A|}\\
&\le \frac{|b^{-1}A_2+A_2||(a_1-a_2)A_1+(b_1-b_2)A_1|}{|A|}\\
&\le \frac{|A_2+bA_2|\Delta^{15}}{|A|^{15}}.
\end{align*}

Since $-2bA_2$ can be covered by at most $|(A-b)^2-A^2-A^2|/|A|$ translates of $-A^2$, we have

\[|-2A_2-2bA_2|\le \frac{|(A-b)^2-A^2-A^2|}{|A|}|-2A_2-A^2|\ll \frac{\Delta^3}{|A|^3}|-2A-A^2|.\]

Moreover, we also have
\[|-A^2-2A|=|-A^2-2A+1|\le |(A-1)^2-A^2-A^2|\le \frac{\Delta^3}{|A|^2}.\]
Therefore
\[|A_2+bA_2|=|-2A_2-2bA_2|\ll \frac{\Delta^6}{|A|^5}.\]
In other words, we obtain
\[\Delta\gg |A|^{1+\frac{1}{21}}.\]
{\bf Case 3:} $A^{-1}\cdot R(A, A)\not\subset R(A, A)$.

As above, in this case, there exist $a_1, a_2, b_1, b_2, b\in A, b\ne 0$ such that
\[r:=b^{-1}\cdot \frac{a_1-a_2}{b_1-b_2}\not\in R(A, A).\]
As in Case $2,$ we see that  $r^{-1}$ exists, and one can use the same argument to show that
\[\Delta\gg |A|^{1+\frac{1}{21}}.\]
{\bf Case 4:} We now consider the last case
\begin{align}
1+R(A, A)&\subset R(A, A)\\
A\cdot R(A, A)&\subset R(A, A)\\
A^{-1}\cdot R(A, A)&\subset R(A, A).
\end{align}
In the next step, we prove that for any polynomial $F(x_1, \ldots, x_n)$ in $n$ variables with integer coefficients, we have
\[F(A, \ldots, A)+R(A, A)\subset R(A, A).\]
Indeed, it is sufficient to prove that
\[1+R(A, A)\subset R(A, A), ~A^m+R(A, A)\subset R(A, A)\]
for any integer $m\ge 1$, and $A^m=A\cdots A$ ($m$ times).

It is clear that the first requirement $1+R(A, A)\subset R(A, A)$ is satisfied.   For the second requirement, it is sufficient to prove it for $m=2$, since one can use inductive arguments for larger $m$.

Let $a, a'$ be arbitrary elements in $A$. We now show that
\[aa'+R(A, A)\subset R(A, A).\]
If either $a=0$ or $a'=0$, then we are done. Thus we may assume that $a\ne 0$ and $a'\ne 0$.

First we have
\[a+R(A, A)=a(1+a^{-1}R(A, A))\subset a(1+R(A, A)))\subset R(A, A),\]
and
\[aa'+R(A, A)=a(a'+a^{-1}R(A, A))\subset a(a'+R(A, A))\subset aR(A, A)\subset R(A, A).\]

In other words, we have proved that for any polynomial $F(x_1, x_2, \ldots, x_n)$ with integer coefficients, we have
\[F(A, \ldots, A)+R(A, A)\subset R(A, A).\]
On the other hand, Lemma \ref{lmt3} gives us that  there exists a polynomial $P$ such that \[P(A, \ldots, A)=\mathbb{F}_A.\] This follows that
\[\mathbb{F}_A+R(A, A)\subset R(A, A).\]

It follows from the assumption of the theorem that
\[|A|=|A\cap \mathbb{F}_A|\le |\mathbb{F}_A|^{1/2}.\]
Hence, $|R(A, A)|\ge |\mathbb{F}_A|\ge |A|^2$.

Next we will show that there exists $r\in R(A, A)$ such that
\[|A+rA|\gg |A|^2.\]

Indeed, let $E^+(X, Y)$ be the number of tuples $(x_1, x_2, y_1, y_2)\in X^2\times Y^2 $ such that
\[x_1+y_1=x_2+y_2.\]
Notice that the sum $\sum_{r\in R(A, A)}E^+(A, rA)$ is the number of tuples $(a_1, a_2, b_1, b_2)\in A^2\times A^2$ such that
\[a_1+rb_1=a_2+rb_2\]
with $a_1, a_2\in A$, $b_1, b_2\in A$ and $r\in R(A, A)$. It is clear that there are at most $|R(A, A)||A|^2$ tuples with $a_1=a_2, b_1=b_2$, and at most $|A|^4$ tuples with $b_1\ne b_2$. Therefore, we get
\[\sum_{r\in R(A, A)}E^+(A, rA)\le |R(A, A)||A|^2+|A|^4\le 2|R(A, A)||A|^2.\]
By the pigeon-hole principle, there exists $r:=\frac{a_1-a_2}{b_1-b_2}\in R(A, A)$ such that
\[E^+(A, rA)\le 2|A|^2.\]
Hence,
\[|A+rA|\ge |A|^2/2.\]
Suppose $r=(a_1-a_2)/(b_1-b_2)$. Let $A_1$ be the set defined as in Case $1$. Note that we can always assume that $|A_1|\ge 9|A|/10$. Hence
\[|A\setminus A_1+rA_1|, |A+r(A\setminus A_1)|\le |A|^2/10.\]
Thus we get
\[|A_1+rA_1|\gg |A|^2.\]
Using the upper bound of $|A_1+rA_1|$ in  Case $1$, we have
\[|A|^2\ll |A_1+rA_1|= |(b_1-b_2)A_1+(a_1-a_2)A_1|\le \frac{\Delta^{15}}{|A|^{14}},\]
which gives us
\[\Delta\gg |A|^{1+\frac{1}{15}}.\]
This completes the proof of the theorem.
\end{proof}
\bigskip
\section{Proof of Theorem \ref{thm2}}
For $A\subset \mathbb{F}_q$ and $B:=A+A$, we define $E(A^2, (A-B)^2)$ as the number of $6$-tuples $(a_1, a_2, b_1, a_3, a_4, b_2)\in (A\times A\times B)^2$ such that
\[a_1^2+(a_2-b_1)^2=a_3^2+(a_4-b_2)^2.\]
\begin{lemma}\label{thm21}
Let $A\subset \mathbb{F}_q$, and $B:=A+A$. If
\[E\left(A^2,(A-B)^2\right)\le |A|^{3-\epsilon}|B|^2,\]
then we have
\[\max\left\lbrace |A+A|, |A^2+A^2|\right\rbrace \gg |A|^{1+\frac{\epsilon}{3}}.\]
\end{lemma}
\begin{proof}
We consider the equation
\begin{equation}\label{eqx9}x^2+(y-z)^2=t,\end{equation}
where $x\in A, y\in B, z\in A, t\in A^2+A^2$.

It is clear that for any triple $(a, b, c)\in A^3$, we have a solution $(a, b+c, c, a^2+b^2)\in A\times B\times A\times (A^2+A^2)$ of
the equation (\ref{eqx9}). By the Cauchy-Schwarz inequality, we have
\[|A|^6\le |A^2+A^2|\cdot E(A^2, (A-B)^2)\le |A^2+A^2||A+A|^2|A|^{3-\epsilon},\]
which implies that
\[\max\left\lbrace |A+A|, |A^2+A^2|\right\rbrace \gg |A|^{1+\frac{\epsilon}{3}}.\]
This concludes the proof of the lemma.
\end{proof}
\bigskip
In this section, without loss of generality, we assume that for any subset $A'\subset A$ with $|A'|\gg |A|$, we have
\[E\left(A'^2,(A'-B)^2\right)\ge |A'|^{3-\epsilon}|B|^2,\]
otherwise, we are done by Lemma \ref{thm21}.
\bigskip
\begin{lemma}\label{lmx}
For $A\subset \mathbb{F}_q,$ set $B=A+A$. Suppose $E\left(A^2,(A-B)^2\right)\ge |A|^{3-\epsilon}|B|^2$. Then there exist  subsets $X\subset A$ and $Y\subset B$ with $|X|\gg |A|^{1-\epsilon}, |Y|\gg |B|^{1-\epsilon}$ such that the following holds:
\begin{itemize}
\item For any $b\in Y,$  $90\%$ of $(A-b)^2$ can be covered by at most $|(A-b_1)^2|^{\epsilon}\sim |A|^{\epsilon}$ translates of $-A^2$.
\item For any $a\in X$, $90\%$ of $(a-B)^2$ can be covered by at most $|A|^{\epsilon}$ translates of $-A^2$.
\end{itemize}
\end{lemma}
\begin{proof}
Since $E\left(A^2,(A-B)^2\right)\ge |A|^{3-\epsilon}|B|^2,$ there exists a set $Y\subset B$ with $|Y|\gg |B|^{1-\epsilon}$ such that for any $b_1\in Y$, the number of $5$-tuples $(a_1, a_2, a_3, a_4, b)\in A^4\times B$ satisfying the equation
\begin{equation}\label{eqx2}a_1^2+(a_2-b_1)^2=a_3^2+(b-a_4)^2\end{equation}
is at least $|A|^{3-\epsilon}|B|$.

We now show that for any $b_1\in Y$, we can cover $90\%$ of $(A-b_1)^2$ by at most $|A|^{\epsilon}$ translates of $-A^2$. It suffices to show that we can find one translate of $-A^2$ such that the intersection of  $(A-b_1)^2$ and that translate is of size at least $|A|^{1-\epsilon}\sim |(A-b_1)^2|^{1-\epsilon}$. When we find such a translate, we remove the intersection and then repeat the process until the size of the remaining part of $(A-b_1)^2$ is less than $|(A-b_1)^2|/10$. 

Indeed, the number of solutions of the equation (\ref{eqx2}) is at least $|A|^{3-\epsilon}|B|$, and thus there exist $b\in B$ and $a_3, a_4\in A$ such that
\[|(A-b_1)^2\cap (-A^2+a_3^2+(b-a_4)^2)|\gg \frac{|A|}{|A|^{\epsilon}}\gg |(A-b_1)^2|^{1-\epsilon}.\]
Hence, there is a translate of $-A^2$ such that it intersects $(A-b_1)^2$ in at least $\gg |(A-b_1)^2|^{1-\epsilon}$ elements.

In the next step, we are going to show that there is a subset $X$ of $A$ with $|X|\gg |A|^{1-\epsilon}$ such that for any $a_4\in X$, we can cover $90\%$ of $(B-a_4)^2$ by at most $|A|^{\epsilon}$ translates of $-A^2$. It suffices to show that we can find one translate of $-A^2$ such that the intersection of  $(B-a_4)^2$ and that translate is of size at least $|B||A|^{-\epsilon}\gg |(B-a_4)^2||A|^{-\epsilon}$. When we find such a translate, we remove the intersection and then repeat the process until the size of the remaining part of $(B-a_4)^2$ is less than $|(B-a_4)^2|/10$. 

Since $E(A^2, (A-B)^2)\gg |A|^{3-\epsilon}|B|^2$, there is a subset $A'\subset A$ with $|A'|\gg |A|^{1-\epsilon}$ such that, for each $a_4\in A'$,  the number of solutions of the equation
\begin{equation}\label{eqx22}a_1^2+(a_2-b_1)^2=a_3^2+(b-a_4)^2\end{equation}
is at least $|A|^{2-\epsilon}|B|^2$. Hence, there exist $a_2, a_1\in A$ and $b_1\in B$ such that 
\[|(-A^2+(a_2-b_1)^2+a_1^2)\cap (B-a_4)^2|\gg \frac{|B|}{|A|^{\epsilon}}.\]

Thus there is a translate of $-A^2$ that intersects with $(B-a_4)^2$ in at least $|B|/|A|^{\epsilon}$ elements.
\end{proof}
we now are ready to prove Theorem \ref{thm2}.
\bigskip
\begin{proof}[Proof of Theorem \ref{thm2}]
By employing Lemma \ref{lmt2}, without loss of generality, we can suppose that $A$ satisfies the following inequality 
\[|A^2+A^2+A^2|\ll \frac{|A^2+A^2|^2}{|A|}.\]

Let $\epsilon>0$ be a parameter which will be chosen at the end of the proof.  Let $X$ and $Y$ be sets defined as in Lemma \ref{lmx}. For the simplicity, we assume that $|X|=|A|^{1-\epsilon}$ and $|Y|=|A-A|^{1-\epsilon}$. As in the proof of Theorem \ref{thm1}, we first define the ratio set:
\[R(X, Y):=\left\lbrace  \frac{b_1-b_2}{a_1-a_2}\colon a_1, a_2\in X,  b_1, b_2\in Y\right\rbrace.\]
We now consider the following cases:

{\bf Case 1:} $1+R(X, Y)\not\subset R(X, Y)$.

In this case, there exist $a_1, a_2\in X, b_1, b_2\in Y$ such that
\[r:=1+\frac{b_1-b_2}{a_1- a_2}\not\in R(X, Y).\]

Applying Lemma \ref{lmx}, we can find subsets $X_1\subset X$ and $Y_1\subset Y$ with $|X_1|\gg |X|$, $|Y_1|\gg |Y|$ such that $(X_1-b_1)^2, (X_1-b_2)^2, (Y_1-a_1)^2, (Y_1-a_2)^2$ can be covered by at most $|A|^\epsilon$ translates of $-A^2$.

One can apply Lemma \ref{cover} four times to obtain subsets $X_2\subset X_1, Y_2\subset Y_1$ with $|X_2|\gg |X_1|, |Y_2|\gg |Y_1|$ such that $2a_1Y_2$ can be covered by at most
\[\frac{|2a_1Y_2+A^2-a_1^2|}{|A|}\le \frac{|(Y_2-a_1)^2-A^2-A^2|}{|A|}\]
translates of $A^2$,  $-2a_2Y_2$ can be covered by at most
\[\frac{|-2a_1Y_2-A^2+a_2^2|}{|A|}\le \frac{|(Y_2-a_2)^2-A^2-A^2|}{|A|}\]
translates of $A^2$,  $-2b_2X_2$ can be covered by at most
\[\frac{|-2b_2X_2-A^2+b_2^2|}{|A|}\le \frac{|(X_2-b_2)^2-A^2-A^2|}{|A|}\]
translates of $A^2$, and $2b_1X_2$ can be covered by at most
\[\frac{|2b_1X_2+A_1^2-b_1^2|}{|A_1|}\le \frac{|(X_2-b_1)^2-A^2-A^2|}{|A|}\]
translates of $A_1^2$, where $A_1\subset A$ with $|A_1|\gg |A|$ and
\[|A_1^2+A^2+A^2+A^2|\ll \frac{|A^2+A^2|^3}{|A|^2},\]
which can be obtained by using Lemma \ref{lmt2}.

Applying Lemma \ref{lmt2} again, we see that there exists a subset $Y_3\subset Y_2$ such that $|Y_3|\gg |Y_2|$ and
\begin{align}\label{eq-o-1}
|(a_1-a_2)X_2+(b_1-b_2)X_2+(a_1-a_2)Y_3| &\ll \frac{|X_2+Y_2||(b_1-b_2)X_2+(a_1-a_2)Y_2|}{|Y_2|}\nonumber\\
&\ll \frac{|A+A+A||(b_1-b_2)X_2+(a_1-a_2)Y_2|}{|Y|}\nonumber\\
&\ll \frac{|A+A|^3}{|A|^2}\cdot \frac{|(b_1-b_2)X_2+(a_1-a_2)Y_2|}{|Y|}.
\end{align}
On the other hand, we also have 
\begin{equation}\label{eq-o-2}|(a_1-a_2)X_2+(b_1-b_2)X_2+(a_1-a_2)Y_3| \geq |rX_2+Y_3|.\end{equation}
Since $r\not\in R(X, Y)$, the equation
\[a_1-a_2=r(b_1+rb_2)\]
has no non-trivial solutions.  This gives us
\begin{equation}\label{eq-o-3}|rX_2+Y_3|=|X_2||Y_3|\gg |X||Y|. \end{equation}
We now estimate $(b_1-b_2)X_2+(a_1-a_2)Y_2$ as follows.

First we note that
\[|(b_1-b_2)X_2+(a_1-a_2)Y_2|=|2b_1X_2-2b_2X_2+2a_1Y_2-2a_2Y_2|.\]

Since $2b_1X_2$ can be covered by at most $|(X_1-b_1)^2+A^2-A^2|/|A|$ copies of $A_1^2$, $2b_2X_1$ can be covered by at most $|(X_1-b_2)^2+A^2-A^2|/|A|$ copies of $A^2$, $2a_1Y_1$ can be covered by at most $|(Y_1-a_1)^2+A^2-A^2|/|A|$ copies of $A^2$, and $2a_2Y_1$ can be covered by at most $|(Y_1-a_2)^2+A^2-A^2|/|A|$ copies of $A^2$, we have
\begin{align}\label{eq-o-4}
&|(b_1-b_2)X_2+(a_1-a_2)Y_2|\ll \\&\ll \frac{|(X_1-b_1)^2-A^2-A^2||(X_1-b_2)^2-A^2-A^2|}{|A|^4}\nonumber\\&\times |(Y_1-a_1)^2-A^2-A^2||(Y_1-a_2)^2-A^2-A^2| |A_1^2+A^2+A^2+A^2|\nonumber\\
&\le  \frac{|A^2+A^2|^3}{|A|^{6-4\epsilon}} |-A^2-A^2-A^2|^4\nonumber\\
&\le \frac{|A^2+A^2|^{11}}{|A|^{10-4\epsilon}},\nonumber
\end{align}
where we have used the fact that $(X_1-b_1)^2, (X_1-b_2)^2, (Y_1-a_1)^2, (Y_1-a_2)^2$ can be covered by at most $|A|^\epsilon$ translates of $-A^2$.

Putting (\ref{eq-o-1}-\ref{eq-o-4}) together, we obtain
\[|A+A|^3|A^2+A^2|^{11}\gg |A|^{15-5\epsilon}.\]
{\bf Case 2:} $Y \cdot R(X, Y)\not\subset R(X, Y)$.
Similarly, in this case, there exist $a_1, a_2 \in X,  b, b_1, b_2\in Y$ such that
\[r:=b\cdot \frac{b_1-b_2}{a_1-a_2}\not\in R(X, Y).\]
Since $0\in R(X, Y)$, we see that $b\ne 0$, and $b_1\ne b_2$. This tells us that $r^{-1}$ exists.

Let $X_2$ and $Y_2$ be sets defined as in Case $1$.

We use Lemma \ref{cover} to obtain a set $X_3\subset X_2$ with $|X_3|\gg |X_2|$ such that $2bX_3$ can be covered by at most $|(X_3+b)^2-A^2-A^2|/|A|$ translates of $|A|^2$.

Moreover, one also has
\begin{align}\label{eq14142}
|X||Y|\ll |rX_3+Y_2|&\ll  \frac{|X_2+bX_3||(a_1-a_2)Y_2+(b_1-b_2)X_2|}{|X|}\\
&\ll |X_2+bX_3|\cdot \frac{|A^2+A^2|^{11}}{|A|^{10-4\epsilon}|X|}\nonumber.
\end{align}

Since $-2bX_3$ can be covered by at most $|(X_3-b)^2-A^2-A^2|/|A|$ translates of $-A^2$, we have

\[|X_2+bX_3|=|-2X_2-2bX_3|\ll \frac{|(X_3-b)^2-A^2-A^2|}{|A|}\cdot |-2X_2-A^2|\le \frac{|A^2+A^2+A^2|}{|A|^{1-\epsilon}} \cdot |-2X_2-A^2|,\]
where we used the fact that $(X_3-b)^2$ can be covered by at most $|A|^\epsilon$ translates of $-A^2$. Note that it follows from the proof of Lemma \ref{lmx} that we can assume that $1\in Y$ by scaling the set $A$. Therefore, we can bound $|A^2-2X_2|$ as follows
\[|-A^2-2X_2|\ll |-A^2-2A+1|\le |(A-1)^2-A^2-A^2|\le |A|^{\epsilon} |A^2+A^2+A^2|.\]
In other words, we have indicated that
\begin{equation}\label{eq14143}|X_2+bX_3|\ll \frac{|A^2+A^2+A^2|^2}{|A|^{1-2\epsilon}}\ll \frac{|A^2+A^2|^4}{|A|^{3-2\epsilon}},\end{equation}
since we have assumed that $|A^2+A^2+A^2|\ll |A^2+A^2|^2/|A|$.  Putting (\ref{eq14142}) and (\ref{eq14143}) together, we obtain
\[|A^2+A^2|^{15}\gg |A|^{16-9\epsilon}.\]

{\bf Case 3:} $Y^{-1}\cdot R(X,  Y)\not\subset R(X, Y)$.

As above, in this case, there exist $a_1, a_2 \in X,  b_1, b_2, b\in Y, b\ne 0$ such that
\[r:=b^{-1}\cdot \frac{b_1-b_2}{a_1-a_2}\not\in R(X, Y).\]
As in Case $2$, we see that $r^{-1}$ exists, and one can use the same argument to show that
\[|A^2+A^2|^{15}\gg |A|^{16-9\epsilon}.\]
{\bf Case 4:} We now consider the last case
\begin{align}
1+R(X, Y)&\subset R(X, Y)\\
Y\cdot R(X, Y)&\subset R(X, Y)\\
Y^{-1}\cdot R(X, Y)&\subset R(X, Y).
\end{align}
In the next step, we prove that for any polynomial $F(x_1, \ldots, x_n)$ in $n$ variables with integer coefficients, we have
\[F(Y, \ldots, Y)+R(X, Y)\subset R(X, Y).\]
Indeed, it is sufficient to prove that
\[1+R(X, Y)\subset R(X, Y), ~Y^m+R(X, Y)\subset R(X, Y),\]
for any integer $m\ge 1$, and $Y^m=Y\cdots Y$ ($m$ times).

It is clear that the first requirement $1+R(X, Y)\subset R(X, Y)$ is satisfied.   For the second requirement, it is sufficient to prove it for $m=2$, since one can use inductive arguments for larger $m$.

Let $y, y'$ be arbitrary elements in $Y$. We now show that
\[yy'+R(X, Y)\subset R(X, Y).\]
If either $y=0$ or $y'=0$, then we are done. Thus we can assume that $y\ne 0$ and $y'\ne 0$.

First we have
\[y+R(X, Y)=y(1+y^{-1}R(X, Y))\subset y(1+R(X, Y))\subset R(X, Y),\]
and
\[yy'+R(X, Y)=y(y'+y^{-1}R(X, Y))\subset y(y'+R(X, Y))\subset yR(X, Y)\subset R(X, Y).\]

In other words, we have proved that for any polynomial $F(x_1, x_2, \ldots, x_n)$ in some variables with integer coefficients, we have
\[F(Y, \ldots, Y)+R(X, Y)\subset R(X, Y).\]
On the other hand, Lemma \ref{lmt3} gives us that  there exists a polynomial $P$ such that \[P(Y, \ldots, Y)=\mathbb{F}_Y.\] This follows that
\[\mathbb{F}_Y+R(X, Y)\subset R(X, Y).\]

It follows from the assumption of the theorem that
\[|Y|=|Y\cap \mathbb{F}_Y|\le |\mathbb{F}_Y|^{1/2}.\]
Hence, $|R(X, Y)|\ge |\mathbb{F}_Y|\ge |Y|^2$.

Next we will show that there exists $r\in R(X, Y)$ such that either
\[|Y+rX|\ge |Y||X|/2,\]
or
\[|Y+rX|\ge |Y|^2/2.\]
Recall that the sum $\sum_{r\in R(X, Y)}E^+(Y, rX)$ is the number of tuples $(a_1, a_2, b_1, b_2)\in X^2\times Y^2$ such that
\[b_1+ra_1=b_2+ra_2\]
with $a_1, a_2\in X$, $b_1, b_2\in Y$ and $r\in R(X, Y)$. It is clear that there are at most $|R(X, Y)||X||Y|$ tuples with $a_1=a_2, b_1=b_2$, and at most $|X|^2|Y|^2$ tuples with $b_1\ne b_2$. Therefore, we get
\[\sum_{r\in R(X, Y)}E^+(rX, Y)\le |R(X, Y)||X||Y|+|X|^2|Y|^2\le |R(X, Y)||X||Y|+|X|^2|R(X, Y)|.\]
Hence, there exists $r\in R(X, Y)$ such that either $E^+(rX, Y)\le 2|X||Y|$ or $E^+(rX, Y)\le 2|X|^2$. This implies that either
\[|Y+rX|\ge |Y||X|/2,\]
or
\[|Y+rX|\ge |Y|^2/2.\]
Put $r=(b_1-b_2)/(a_1-a_2)$. Let $X_2$ and $Y_2$ be sets defined as in  Case $1$. Note that we can always assume that $|X_2|\ge 9|X|/10$ and $|Y_2|\ge 9|Y|/10$. Thus
\[|Y+r(X\setminus X_2)|+|(Y\setminus Y_2)+rX_2|\le |X||Y|/5.\]

It follows from our assumption that $|X|=|A|^{1-\epsilon}$ and $|Y|=|A-A|^{1-\epsilon}$, we can assume that 
\[|Y_2+rX_2|\gg |X||Y|,\]
or
\[|Y_2+rX_2|\gg |Y|^2.\]

As in Case $1$, we have
\[|Y_2+rX_2|\ll \frac{|A^2+A^2|^{11}}{|A|^{10-4\epsilon}}.\]

In short, we have

\[|A^2+A^2|\gg |A|^{1+\frac{1-6\epsilon}{11}}.\]

Choose $\epsilon=3/42$, the theorem follows directly from Cases ($1$)-($3$) and Lemma \ref{thm21}.
\end{proof}

\section*{Acknowledgments}
D. Koh was supported by Basic Science Research Program through the National
Research Foundation of Korea(NRF) funded by the Ministry of Education, Science
and Technology(NRF-2018R1D1A1B07044469). T. Pham was supported by Swiss National Science Foundation grant P2ELP2175050. C-Y Shen was supported in part by MOST, through grant 104-2628-M-002-015 -MY4.
 

\begin{thebibliography}{00}
\bibitem{AMRS}
E. Aksoy Yazici, B. Murphy, M. Rudnev, and I. Shkredov,
{\em Growth estimates in positive characteristic via collisions},
to appear in International Mathematics Research Notices.
Also in {\tt arXiv:1512.06613} (2015).
      \bibitem{bkt}
      J. Bourgain, N. Katz and T. Tao, \textit{A sum-product estimate in finite fields and their applications}, Geom. Func. Anal. \textbf{14} (2004), 27--57.
	\bibitem{a3}
	J. Bourgain, \textit{More on the sum-product phenomenon in prime fields and its applications}, Int.
J. Number Theory \textbf{1} (2005), no. 1, 1--32.
	\bibitem{es}
	 P. Erd\H{o}s and E. Szemer\'{e}di, \textit{On sums and products of integers}, Studies in Pure Mathematics. To the memory of Paul Turan, Basel: Birkh\"{a}user Verlag, pp. 213-218, 1983.
	 \bibitem{e1}
	  G. Elekes, \textit{On the number of sums and products}, Acta Arith. \textbf{81} (1997), 365–-36.
	  \bibitem{chang1}
	  M. Chang, \textit{The Erd\H{o}s-Szemer\'{e}di problem on sum set and product set}, Annals of mathematics (2003): 939--957.
	  \bibitem{chang2}
	  M. Chang,  \textit{New results on the Erd\H{o}s–Szemerédi sum-product problems}, Comptes Rendus Mathematique \textbf{336}(3) (2003): 201--205.
	\bibitem{heg1}
	N. Hegyv\'{a}ri,  A. S\'{a}rk\"{o}zy, \textit{On Hilbert cubes in certain sets}, The Ramanujan Journal \textbf{3}(3) (1999), 303--314.
	\bibitem{heg2}
	N. Hegyv\'{a}ri,  A. S\'{a}rk\"{o}zy, \textit{A note on the size of the set} $ {A^ 2+ A} $, The Ramanujan Journal \textbf{46}(2) (2018), 357--372.
	\bibitem{heg3}
	N. Hegyv\'{a}ri, F. Hennecart, \textit{Explicit constructions of extractors and expanders}, Acta Arithmetica \textbf{3}(140) (2009), 233--249.
	\bibitem{heg4}
	N. Hegyv\'{a}ri, F. Hennecart,\textit{ Conditional expanding bounds for two-variable functions over prime fields},  European Journal of Combinatorics \textbf{34}(8) (2013), 1365--1382.
	\bibitem{a1}
	N. Hegyv\'{a}ri, F. Hennecart, \textit{Explicit constructions of extractors and expanders,} Acta Arithmetica \textbf{3}(14) (2009): 233--249.
	\bibitem{han}
	Hanson, Brandon, \textit{The additive structure of cartesian products spanning few distinct distances}, Combinatorica (2017): 1--6.
	\bibitem{A}
	A. Iosevich, D. Koh, T. Pham, C-Y. Shen, L. Vinh, A\textit{ new bound on Erd\H{o}s distinct distances problem in the plane over prime fields}, arXiv: 1804.05451v2 [math.CO] 22 May 2018.
	\bibitem{e3}
S.V. Konyagin, I.D. Shkredov, \textit{New results on sum–products in} $\mathbb{R}$, Proc. Steklov Inst.
Math., \textbf{294}(78), (2016), 87--98.
        	\bibitem{shen}
	N. Katz, C-Y. Shen, \textit{A slight improvement to Garaev's sum product estimate}, Proceedings of the American Mathematical Society \textbf{136}(7) (2008), 2499--2504.
				\bibitem{lili} L. Li and O. Roche-Newton, \textit{An improved sum-product estimate for general finite fields}, SIAM J. Discrete Math. \textbf{25} (2011), no.3, 1285--1296.

	\bibitem{R}
M. Rudnev,
{\em On the number of incidences between points and planes in three dimensions}, 
to appear in Combinatorica.
Also in 
{\tt arXiv:1407.0426} (2014).
	
	\bibitem{e4}
	M. Rudnev, I. D. Shkredov, S. Stevens, \textit{On the energy variant of the sum-product conjecture}, arXiv:1607.05053 (2016).
	\bibitem{e2}
	J. Solymosi, \textit{Bounding multiplicative energy by the sumset}, Adv. Math. \textbf{222}(2) (2009),
402--408.
	\bibitem{sh1}
		G. Shakan, \textit{On higher energy decomposition and the sum–product phenomenon}, To appear in Q. J. Math. arXiv:1803.04637 (2018).
		\bibitem{a2}
		Shachar Lovett. Additive combinatorics and its applications in theoretical computer science, Theory of computing library graduate surveys, \textbf{887} (2016), pp. 1--53.
	\bibitem{sha}
	G. Shakan, I. D. Shkredov, 
\textit{Breaking the $6/5$ threshold for sums and products modulo a prime}, arXiv:1806.07091, 2018. 
\bibitem{sol}
J. Solymosi, \textit{Bounding multiplicative energy by the sumset}, Adv. Math. \textbf{222} (2) (2009), 402--408  .

	\bibitem{shk}
	I. D. Shkredov, \textit{Some remarks on sets with small quotient set}, arXiv:1603.04948 (2016).
	\bibitem{s-ds}
	I. D. Shkredov, \textit{Difference sets are not multiplicatively closed},
Discrete Analysis, \textbf{17}(2016), 21 pp.
\bibitem{P}
G. Petridis,
{\em Pinned algebraic distances determined by Cartesian products in $\F^2$},
{\tt arXiv:1610.03172} (2016).
	\bibitem{pham}
	T. Pham, L. A. Vinh, F. De Zeeuw, \textit{Three-variable expanding polynomials and higher-dimensional distinct distances}, Combinatorica, to appear (2018).
	\bibitem{vt}
	V. Vu and T. Tao, Additive Combinatorics, Cambridge Studies in Advanced Mathe-matics.
	
	\end{thebibliography}
\end{document}